\documentclass[a4paper,abstracton]{scrartcl}
\usepackage[utf8]{inputenc}
\usepackage{amsfonts,amsthm,amssymb,amsmath}
\usepackage[inline]{enumitem}
\usepackage{color}
\usepackage{graphicx} 
\usepackage{authblk}

\newtheorem{thm}{Theorem}

\newtheorem{conj}[thm]{Conjecture}

\theoremstyle{remark}

\title{A new problem related to Eulerian graphs}
\author{Marcin Stawiski\footnote{ Corresponding author; stawiski@agh.edu.pl}
}

\affil{AGH University of Krakow,\\ Faculty of Applied Mathematics, \protect\\al. Mickiewicza 30, 30-059 Krakow, Poland}

\begin{document}
\maketitle 
\begin{abstract}
Let $G$ be a graph, and $H$ be a finite subgraph of $G$. We say that $H$ is a (semi) $S$-Eulerian subgraph if there exists a closed (open) trail $T$ in $G$ such that each edge of $H$ appears in $T$. We show that the problem of determining whether a subgraph $H$ of a finite graph $G$ is (semi) $S$-Eulerian is NP-Complete. Moreover, we show that both versions of the problem become linear in time if we restrict ourselves to connected subgraphs $H$.
\end{abstract}
\section{Introduction}

Euler's Theorem, which states that a connected finite graph has an Eulerian Circuit if and only if every its vertex has even degree, is one of the earliest results in graph theory. This theorem was stated by Euler \cite{euler} in 1736. Despite the name, it was not proved by Euler, although he showed the necessity of the condition on degrees of the graph. The first proof of Euler's Theorem is attributed to Hierholzer, and it appeared in the paper \cite{hierholzer} published posthumously in 1873.

In this paper, we begin a study of a new problem related to Eulerian graphs. Let $H$ be a subgraph of a graph $G$. We say that an closed (oped) trail $T$ in $G$ is an \emph{$S$-Eulerian Circuit (Trail)} if every edge of $H$ appears in $T$. We say that $H$ is \emph{(semi) $S$-Eulerian} if it has a $S$-Eulerian Circuit (Trail). Equivalently, a subgraph $H$ of $G$ is (semi) $S$-Eulerian if any only if there exists a (semi) Eulerian graph $H'$ such that $H\subseteq H' \subseteq G$.

The main purpose of this paper is to determine the computational complexity of the problem of determining whether a given subgraph $H$ of a given finite graph $G$ is (semi) $S$-Eulerian for some classes of graphs $H$.
 In particular, we prove the following theorem for the class of all finite graphs $H$.

\begin{thm}\label{main:1}
    The problem of determining whether a given subgraph $H$ of a finite graph $G$ is (semi) $S$-Eulerian  is NP-Complete.
\end{thm}

On the other hand, if we restrict ourselves to connected subgraphs $H$, then the problem becomes solvable in linear time.

\begin{thm}\label{main:2}
The problem of determining whether a given connected subgraph $H$ of a finite graph $G$ is (semi) $S$-Eulerian is linear in time.
\end{thm}

Our proof of the above theorem is constructive and provides an effective algorithm for finding a $S$-Eulerian Circuit (Trail).

Further research could involve studying other classes of graphs $H$ and $G$ for which the problems are linear, polynomial or NP-Complete. As a starting point, we propose the following conjecture.

\begin{conj}
Let $k$ be a natural number. The problem of determining whether a given subgraph $H$ with at most $k$ components of a finite graph $G$ is $S$-Eulerian (semi $S$-Eulerian) is polynomial in time.
\end{conj}
\section{Main results}

The $S$-Eulerian problem is a generalisation not only of the problem of Eulerian graphs but also of the following problem: Given any graph, we want to check if there exists its spanning Eulerian subgraph. To see that this is indeed a generalisation, it suffices to pick $H$ as the graph on all vertices of $G$ without any edge. Pulleyblank \cite{Pulleyblank} showed that the problem of the existence of a spanning Eulerian subgraph is NP-Complete. It follows that the $S$-Eulerian problem is also NP-Complete. We give a proof of NP-Completeness of the problem of the spanning semi-Eulerian subgraphs, which is very similar to the one of Pulleyblank's Theorem, but we include it for completeness. This together gives us a proof of Theorem \ref{main:1}. Pulleyblank's proof makes use of the following theorem.

\begin{thm}[Garey, Johnson, Stockmeyer 1976 \cite{Garey}]\label{thm:garey}
    The Hamiltonian Cycle problem is NP-Complete for the class of subcubic finite graphs.
\end{thm}

We show the following analogue of this problem.

\begin{thm}
    The Hamiltonian Path problem is NP-Complete for the class of subcubic finite graphs.
\end{thm}
\begin{proof}
    First, consider the problem of finding a Hamiltonian Path in a subcubic graph between given two vertices. We can transform Hamiltonian Cycle problem for subcubic graph into this problem in the following way. First, we delete one edge $vu$ of the given graph, and we check whether there exists a Hamiltonian Path in this graph with ends $v$ and $u$. We repeat this for every edge of the graph. It follows that the problem of the Hamiltonian Path between two given vertices in the class of subcubic finite graphs is NP-Complete by Theorem \ref{thm:garey}.

    Now, the Hamiltonian Path problem is equivalent to checking the existence of a Hamiltonian Path between two given vertices for all possible pairs of vertices. Hence, the Hamiltonian Path problem in the class of subcubic finite graphs is NP-Complete.
\end{proof}

Now we proof the analogue of Pulleyblank's Theorem for semi-Eulerian subgraphs.

\begin{thm}
    The problem of the existence of a spanning semi-Eulerian subgraph in the class of subcubic finite graphs is NP-Complete.
\end{thm}
\begin{proof}
    It is enough to notice that a Hamiltonian Path in a subcubic graph is the same as a spanning Eurelian Trail.
\end{proof}

\qed

Now, we prove that the (semi) $S$-Eulerian problem for connected graphs $H$ can be solved in linear time.

\noindent\textit{Proof of Theorem \ref{main:2}.}
Let $S$ be the set of vertices of $H$ of odd degree in $H$. We construct the graph $F$ in the following way. We define $V(F)=S$, and join two distinct vertices of $F$ by an edge if and only if there exists a path that joins them in $G-E(H)$. It is important that we can find $F$ in linear time, as it is enough to find the components of $G$ and count the vertices of each component.   Checking whether there exists a perfect matching in $F$ is trivial, because each component of $F$ is a clique. A perfect matching in $F$ exists if and only if every component of $F$ has an even order. Therefore, we can check it in linear time. We claim that $H$ is $S$-Eulerian if and only if there exists a perfect matching in $F$.

 Assume first that no such matching exists. If we add a path in $E(G)-E(H)$ between the vertices in $H$, then it either decreases the number of odd-degree vertices in a new graph by two, increases this number by two, or leaves it the same. Therefore, each component of $F$ must have even order, as otherwise we cannot lower the number of vertices of odd degree to $0$. This means that we cannot extend $H$ to an Eulerian graph $H'\subseteq G$.
 
 Now, assume that there exists a perfect matching $M$ in $F$. We show how to obtain an Eulerian graph $H'\subseteq G$ containing $H$ using $M$. Let $(e_i \colon i\leq k )$, for some natural number $k$ be an enumeration of the edges of $M$. Starting from $i=0$, we join two endpoints of $e_i$ by a path in $G-E(H)$. If $i=0$, then we do not perform additional actions and proceed with the next index. Assume then that $i\geq 1$, and that after step $i-1$ we have chosen $i$ pairs of different vertices of the set $\bigcup\{e_j\colon j\leq i-1\}$ joined by edge-disjoint paths in $G-E(H)$. We define graph $H_i$ as the graph obtained from $H$ by adding these paths. In step $i$, we join two endpoints of $e_i$ with a path in $G-E(H)$ to obtain the graph $H'_{i+1}$ from the graph $H_{i}$. If this path is edge-disjoint with all the paths chosen in the previous step, then we proceed with a new index. If not, then the new path has common edges with some of the chosen paths. We can assume that if there are two common edges of two paths, then all the edges in between are also their common edges. Now, it is enough to remove the common edges of these paths from the graph $H_i'$. After this operation, we obtain $i+1$ pairwise edge-disjoint paths joining pairs of different vertices from $\bigcup\{e_j\colon j\leq i\}$.

 A proof for the semi $S$-Eulerian problem is basically the same but instead of a perfect matching in $F$ we need a matching between all but 2 vertices of $F$.
 \qed

\section*{Acknowledgement}

The author expresses his gratitude to Florian Lehner for providing the reference of Pulleyblank's paper.

\bibliographystyle{abbrv}
\bibliography{lit.bib}

\end{document}